\documentclass[12pt]{amsart}
\usepackage{amsmath,amssymb,amsbsy,amsfonts,latexsym,amsopn,amstext,
                                               amsxtra,euscript,amscd,bm}
                    
\usepackage{url}
\usepackage[colorlinks,linkcolor=blue,anchorcolor=blue,citecolor=blue]{hyperref}
\usepackage{color}

\begin{document}
\newtheorem{problem}{Problem}
\newtheorem{theorem}{Theorem}
\newtheorem{lemma}[theorem]{Lemma}
\newtheorem{claim}[theorem]{Claim}
\newtheorem{cor}[theorem]{Corollary}
\newtheorem{prop}[theorem]{Proposition}
\newtheorem{definition}{Definition}
\newtheorem{question}[theorem]{Question}

\def\cA{{\mathcal A}}
\def\cB{{\mathcal B}}
\def\cC{{\mathcal C}}
\def\cD{{\mathcal D}}
\def\cE{{\mathcal E}}
\def\cF{{\mathcal F}}
\def\cG{{\mathcal G}}
\def\cH{{\mathcal H}}
\def\cI{{\mathcal I}}
\def\cJ{{\mathcal J}}
\def\cK{{\mathcal K}}
\def\cL{{\mathcal L}}
\def\cM{{\mathcal M}}
\def\cN{{\mathcal N}}
\def\cO{{\mathcal O}}
\def\cP{{\mathcal P}}
\def\cQ{{\mathcal Q}}
\def\cR{{\mathcal R}}
\def\cS{{\mathcal S}}
\def\cT{{\mathcal T}}
\def\cU{{\mathcal U}}
\def\cV{{\mathcal V}}
\def\cW{{\mathcal W}}
\def\cX{{\mathcal X}}
\def\cY{{\mathcal Y}}
\def\cZ{{\mathcal Z}}

\def\A{{\mathbb A}}
\def\B{{\mathbb B}}
\def\C{{\mathbb C}}
\def\D{{\mathbb D}}
\def\E{{\mathbb E}}
\def\F{{\mathbb F}}
\def\G{{\mathbb G}}
\def\I{{\mathbb I}}
\def\J{{\mathbb J}}
\def\K{{\mathbb K}}
\def\L{{\mathbb L}}
\def\M{{\mathbb M}}
\def\N{{\mathbb N}}
\def\O{{\mathbb O}}
\def\P{{\mathbb P}}
\def\Q{{\mathbb Q}}
\def\R{{\mathbb R}}
\def\S{{\mathbb S}}
\def\T{{\mathbb T}}
\def\U{{\mathbb U}}
\def\V{{\mathbb V}}
\def\W{{\mathbb W}}
\def\X{{\mathbb X}}
\def\Y{{\mathbb Y}}
\def\Z{{\mathbb Z}}

\def\ep{{\mathbf{e}}_p}
\def\em{{\mathbf{e}}_m}
\def\eq{{\mathbf{e}}_q}

\def\scr{\scriptstyle}
\def\\{\cr}
\def\({\left(}
\def\){\right)}
\def\[{\left[}
\def\]{\right]}
\def\<{\langle}
\def\>{\rangle}
\def\fl#1{\left\lfloor#1\right\rfloor}
\def\rf#1{\left\lceil#1\right\rceil}
\def\le{\leqslant}
\def\ge{\geqslant}
\def\eps{\varepsilon}
\def\mand{\qquad\mbox{and}\qquad}

\def\sssum{\mathop{\sum\ \sum\ \sum}}
\def\ssum{\mathop{\sum\, \sum}}
\def\ssumw{\mathop{\sum\qquad \sum}}

\def\vec#1{\mathbf{#1}}
\def\inv#1{\overline{#1}}
\def\num#1{\mathrm{num}(#1)}
\def\dist{\mathrm{dist}}

\def\fA{{\mathfrak A}}
\def\fB{{\mathfrak B}}
\def\fC{{\mathfrak C}}
\def\fU{{\mathfrak U}}
\def\fV{{\mathfrak V}}

\newcommand{\bflambda}{{\boldsymbol{\lambda}}}
\newcommand{\bfxi}{{\boldsymbol{\xi}}}
\newcommand{\bfrho}{{\boldsymbol{\rho}}}
\newcommand{\bfnu}{{\boldsymbol{\nu}}}

\def\GL{\mathrm{GL}}
\def\SL{\mathrm{SL}}

\def\Hba{\overline{\cH}_{a,m}}
\def\Hta{\widetilde{\cH}_{a,m}}
\def\Hb1{\overline{\cH}_{m}}
\def\Ht1{\widetilde{\cH}_{m}}

\def\flp#1{{\left\langle#1\right\rangle}_p}
\def\flm#1{{\left\langle#1\right\rangle}_m}
\def\dmod#1#2{\left\|#1\right\|_{#2}}
\def\dmodq#1{\left\|#1\right\|_q}

\def\Zm{\Z/m\Z}

\def\Err{{\mathbf{E}}}

\newcommand{\comm}[1]{\marginpar{%
\vskip-\baselineskip 
\raggedright\footnotesize
\itshape\hrule\smallskip#1\par\smallskip\hrule}}
\pagestyle{plain}
\def\xxx{\vskip5pt\hrule\vskip5pt}


\title{\bf Some multiplicative equations in finite fields}

\date{\today}

 \author[B. Kerr] {Bryce Kerr}

\address{School of Physical, Environmental and Mathematical Sciences, The University of New South Wales Canberra, Australia}
\email{b.kerr@adfa.edu.au}

\begin{abstract}
In this paper we consider estimating the number of solutions to multiplicative equations in finite fields when the variables run through certain sets with high additive structure. In particular, we consider estimating the multiplicative energy of generalized arithmetic progressions in prime fields and of boxes in arbitrary finite fields and obtain sharp bounds in more general scenarios than previously known.  Our arguments extend some ideas of Konyagin and Bourgain and Chang into new settings.
\end{abstract}

\maketitle


\maketitle
\section{Introduction}
For a prime number $q$ and integer $n$ consider the finite field $\F_{q^{n}}$ with $q^n$ elements. For a subset $\cA\subseteq \F_{q^n}$ we define the multiplicative energy $E(\cA)$ of $\cA$ to count the number of solutions to the equation
$$a_1a_2=a_3a_4, \quad a_1,a_2,a_3,a_4\in \cA.$$
In this paper we consider estimating $E(\cA)$ for certain sets $\cA$ with large additive structure. In particular, we consider the case of boxes in arbitrary finite fields and generalized arithmetic progressions in prime fields. These two problems may be considered as extreme cases of the sum-product phenomenon of Erd\"{o}s and Szemer\'{e}di~\cite{ES},  established in the setting of prime fields by Bourgain, Katz and Tao~\cite{BKT} and arbitrary finite fields by Katz and Shen~\cite{KS}. The sum-product theorem over $\F_{q^{n}}$ states that for any $\varepsilon$ there exists some $\delta>0$ such that if $|\cA|\le q^{(1-\varepsilon)n}$ then
\begin{align}
\label{eq:sumproduct}
\max\{|\cA\cA|,|\cA+\cA| \}\gg |\cA|^{1+\delta},
\end{align}
with the condition that if $n\ge 2$ then $\cA$ does not have a large intersection with any proper subfield, where $\cA\cA$ and $\cA+\cA$ denote the sum and product set 
$$\cA\cA=\{ a_1a_2 \ : \ a_1,a_2\in \cA \}, \quad \cA+\cA=\{ a_1+a_2 \ : \ a_1,a_2\in \cA\}.$$
 An important factor in this problem is how large one may take $\delta$ in~\eqref{eq:sumproduct}. Erd\"{o}s and Szemer\'{e}di~\cite{ES} conjectured that for any set of integers  $\cA$  one may take any fixed $\delta<1$. We expect this conjecture to remain true over finite fields with suitable size restrictions on $\cA$ and the intersection of $\cA$ with proper subfields. Current techniques are still far from resolving this conjecture and  and we refer the reader to~\cite{RSS},~\cite{Sha} and~\cite{LR,RRS} for the current best  quantitative results for sum product over $\R$,  prime fields and general finite fields. 
\newline

A typical approach to the sum-product problem is to estimate the multiplicative energy of a set $\cA$ in terms of the size of the sumset $\cA+\cA$ since it follows from the Cauchy-Schwarz inequality 
\begin{align*}
|\cA\cA|\ge \frac{|\cA|^4}{E(\cA)}.
\end{align*}
For sets $\cA$ satisfying 
\begin{align}
\label{eq:smalldoubling}
|\cA+\cA|\ll |\cA|,
\end{align}
 we expect that
\begin{align}
\label{eq:multenergysmallsumset}
E(\cA)\ll |\cA|^{2+o(1)},
\end{align}
from which it would follow that 
\begin{align*}
|\cA\cA|\gg |\cA|^{2-\varepsilon}.
\end{align*}
This is known to hold over $\R$ by a result of  Elekes and Ruzsa~\cite{ER}, see also~\cite{Chang2}, although  still open in  the case of finite fields  and we refer the reader to~\cite{MPRRS} for the sharpest results in the setting of small sumset in prime fields.  In this paper we consider the problem of obtaining estimates of the strength~\eqref{eq:multenergysmallsumset} under the condition~\eqref{eq:smalldoubling} in the setting of finite fields and obtain some new instances of when this bound holds.
\newline

 An important class of sets with small sumset are generalized arithmetic progressions, which are defined as sets of the form
$$\cA=\{ c+\alpha_1h_1,\dots,\alpha_dh_d +\beta \ :  \ 1\le h_i\le H_i\},$$
and define $\cA$ to be proper if $|\cA|=H_1\dots H_d$. By Frieman's theorem, see for example~\cite[Chapter~5]{TV}, every set $\cA$ satisfying~\eqref{eq:smalldoubling} is dense in some proper generalized arithmetic progression and hence an approach to extending the result of Elekes and Ruzsa~\cite{ER} into finite fields is to show that~\eqref{eq:multenergysmallsumset} holds for generalized arithmetic progressions.  We take a step forward in this direction and give the expected upper bound for $E(\cA)$ for a certain family of generalized arithmetic progressions, see Theorem~\ref{thm:main2} below. Roughly speaking, our result holds for generalized arithmetic progressions which are smaller portions of proper generalized arithmetic progressions. 
\newline


We also consider estimating the multiplicative energy of boxes in arbitrary finite fields. Let $\omega_1,\dots,\omega_n$ be a basis for $F_{q^{n}}$ as a vector space over $\F_q$ and define the box
$$B=\{ \omega_1 h_1+\dots+\omega_nh_n \ : \ M_i<h_i\le M_i+H_i\}.$$
The first estimates for $E(B)$ were  motivated by the problem of extending the Burgess bound into aribtrary finite fields and are due to Burgess~\cite{Bur1} and Karatsuba~\cite{Kar1,Kar2} although the results of Burgess and Karatsuba are not uniform with respect to the basis $\omega_1,\dots,\omega_n$. Davenport and Lewis~\cite{DL} provided the first estimates uniform with respect to the basis $\omega_1,\dots,\omega_n$ although their bound is quantitatively much weaker than that of Burgess and Karatsuba. The estimate of Davenport and Lewis was improved by Chang~\cite{Chang1} using techniques from additive combinatorics which was further improved by Konyagin~\cite{Kon} who showed the expected upper bound
$$E(B)\ll |B|^{2+o(1)},$$
 in the special case that 
$$H_1=H_2=\dots=H_n,$$
and we note that removing this restriction in Konyagin's argument seems to be a difficult problem. Recently Gabdullin~\cite{Ga} has extended Konyagin's estimate to arbitrary boxes when $n=2,3$. In this paper we show Konyagin's estimate holds with the weaker condition 
$$\max{H_i}\ll q^{1/n}\min{H_i},$$
for arbitrary $n$.
We follow Konyagin's strategy which is based on considering the successive minima of a certain family of lattice and their duals and our main novelty for this section comes from establishing certain inequalities for these successive minima by using Siegel's lemma.
\newline

Finally we draw some comparisions between our argument for generalized arithmetic progressions and Konyagin's approach~\cite{Kon}, further developed by Bourgain and Chang~\cite{BC} to deal with multiplicative equations with systems of linear forms. Both Konyagin and Bourgain and Chang reduce the problem to a lattice point counting problem over a family of lattices. An important feature of these families is that they are in a sense self dual which allows control of the successive minima via transference theorems. In order to reduce the problem of multiplicative energy of generalized arithmetic progressions into a lattice point counting problem with the same symmetry as in~\cite{BC,Kon} we first expand into additive characters and considering the sets of large Fourier coefficients, this allows a reduction of the problem into multiplicative equations with generalized arithmetic progressions and Bohr sets and this form of the problem has suitable symmetry.

\section{Main results}
\begin{theorem}
\label{thm:main1}
Let $q$ be prime, $n$ a positive integer and suppose $\omega_1,\dots,\omega_n$ is a basis for $\F_{q^{n}}$ as a vector space over $\F_{q}$. For two $n$-tuples of positive integers 
$H=(H_1,\dots,H_n)$ and $M=(M_1,\dots,M_n)$ we let $B$ denote the box
$$B=\{ \omega_1 h_1+\dots+\omega_nh_n \ : \ M_i<h_i\le M_i+H_i\}.$$
If $H_1,\dots,H_n$ satisfy
$$H_n\le H_{n-1}\le \dots \le H_1\le q,$$
\begin{align}
\label{eq:Hcondthm1}
\prod_{k=1}^{i-1}H_k\ll \frac{qH_i^{i}}{H_{i-1}}, \quad 2\le i \le n,
\end{align}
and 
\begin{align}
\label{eq:Hcondthm2}
H_{n-i+1}^{i}\ll qH_n\prod_{k=n-i+2}^{n}H_k, \quad 2\le i \le n,
\end{align}
then we have 
\begin{align*}
E(B)\ll  \frac{|B|^4}{q^{n}}+|B|^2(\log{|B|})^{n}.
\end{align*}
\end{theorem}
We may put the conditions on $H_1,\dots,H_n$  occuring in Theorem~\ref{thm:main1} in the following simpler form.
\begin{cor}
\label{cor:main1}
Let $q$ be prime, $n$ a positive integer and suppose $\omega_1,\dots,\omega_n$ is a basis for $\F_{q^{n}}$ as a vector space over $\F_{q}$. For two $n$-tuples of positive integers 
$H=(H_1,\dots,H_n)$ and $M=(M_1,\dots,M_n)$ we let $B$ denote the box
$$B=\{ \omega_1 h_1+\dots+\omega_nh_n \ : \ M_i<h_i\le M_i+H_i\}.$$
If $H_1,\dots,H_n$ satisfy
$$H_n\le H_{n-1}\le \dots \le H_1\le q,$$
and
\begin{align*}
H_1\ll q^{1/n}H_n,
\end{align*}
then we have 
\begin{align*}
E(B)\ll  \frac{|B|^4}{q^{n}}+|B|^2(\log{|B|})^{n}.
\end{align*}
\end{cor}
We next consider estimating the multiplicative energy of generalized arithmetic progressions in prime fields.
\begin{theorem}
\label{thm:main2}
Let $q$ be a prime number, $\mathcal{A}\subset \F_q$  a generalised arithmetic progression given by 
\begin{align*}
\cA=\{ \alpha_1h_1+\dots+\alpha_dh_d  \ : \ 1\le h_i \le H\},
\end{align*}
and suppose that the progression
$$\cA'=\{ \alpha_1h_1+\dots+\alpha_dh_d \ : \ |h_i| \le H^2\},$$
is proper. 
Then we have
$$E(\cA)\ll |\mathcal{A}|^2(\log{H})^{2d+1}.$$
\end{theorem}
Theorem~\ref{thm:main2} implies the same estimate with arbitrary translates of $\cA$
\begin{cor}
\label{cor:main2}
Let $q$ be a prime number, $\mathcal{A}\subset \F_q$  a generalized arithmetic progression given by 
\begin{align*}
\cA=\{ \alpha_1h_1+\dots+\alpha_dh_d+\beta  \ : \ 1\le h_i \le H\},
\end{align*}
and suppose that the progression
$$\cA'=\{ \alpha_1h_1+\dots+\alpha_dh_d \ : \ |h_i| \le H^2\},$$
is proper. 
Then we have
$$E(\cA)\ll |\mathcal{A}|^2(\log{H})^{2d+1}.$$
\end{cor}
Removing the condition of equal side lengths in Corollary~\ref{cor:main2} may be a difficult problem although we note to obtain an estimate of the form 
\begin{align*}
E(\cA)\ll |\cA|^{2+o(1)},
\end{align*}
valid for aribtrary proper generalized arithmetic progression it is sufficent to replace the condition $\cA'$ is proper with $\cA$ is proper. For example, supposing $\cA$ is of the form
\begin{align*}
\cA=\{ \alpha_1h_1+\dots+\alpha_dh_d \ : \ 1\le h_i\le H_i\},
\end{align*}
choosing $H$ sufficiently small in terms of $H_1,\dots,H_d$ and partitioning each $1\le h_i \le H_i$ into
$$h_i=h_{0,i}+h_{1,i}H+\dots+h_{\ell,i}H^{\ell}, \quad 0\le h_{j,i}<H,$$
allows for the reduction to the case of generalized arithmetic progressions of equal side length.
As a consequence of Theorem~\ref{thm:main1} and Corollary~\ref{cor:main2} we have the following.
\begin{cor}
Let $B,\cA$ be as in  Theorem~\ref{thm:main1} and Corollary~\ref{cor:main2} and suppose $|B|\le p^{n/2}$. For any $\varepsilon>0$ we have
\begin{align*}
|B B|\gg |B|^{2-\varepsilon}, 
\end{align*}
and 
\begin{align*}
|\cA \cA| \gg |A|^{2-\varepsilon}.
\end{align*}
\end{cor}

\section{Background from the geometry of numbers}

 The following is Minkowski's second theorem, for a proof see~\cite[Theorem~3.30]{TV}.
\begin{lemma}
\label{lem:mst}
Suppose $\Gamma \subseteq \R^{d}$ is a lattice, $D\subseteq \R^{d}$ a convex body and let $\lambda_1,\dots,\lambda_d$ denote the successive minima of $\Gamma$ with respect to $D$. Then we have
$$\frac{\text{Vol}(D)}{\text{Vol}(\R^d/\Gamma)}\ll\lambda_1\dots\lambda_d  \ll\frac{\text{Vol}(D)}{\text{Vol}(\R^d/\Gamma)}.$$
\end{lemma}
 For a proof of the following,  see~\cite[Proposition~2.1]{BHM}.
\begin{lemma}
\label{lem:latticesm}
Suppose $\Gamma \subseteq \R^{d}$ is a lattice, $D\subseteq \R^{d}$ a convex body and let $\lambda_1,\dots,\lambda_d$ denote the successive minima of $\Gamma$ with respect to $D$. Then we have
$$|\Gamma \cap D|\ll \prod_{j=1}^{d}\max\left(1, \frac{1}{\lambda_j} \right).$$
\end{lemma}
For a lattice $\Gamma$ and  a convex body $D$  we define the dual  lattice $\Gamma^*$ and dual body $D^*$ by
$$\Gamma^*=\{ x\in \R^{d} : \langle x,y \rangle \in \Z \quad \text{for all} \quad y\in \Gamma\},$$
$$D^{*}=\{ x\in \R^{d} : \langle x,y \rangle \le 1 \quad \text{for all} \quad y\in D\}.$$

 The following transference principle is due to Mahler~\cite{Mah},  see also~\cite{Ba} for sharper implied constants.
\begin{lemma}
\label{lem:transfer}
Let $\Gamma\subset \R^{d}$ be a lattice, $D\subseteq \R^{d}$ a convex body and let $\Gamma^{*}$ and $D^{*}$ denote the dual lattice and dual body. Let $\lambda_1,\dots,\lambda_d$ denote the  successive minima of $\Gamma$ with respect to $D$ and $\lambda_1^{*},\dots,\lambda_d^{*}$ the successive minima of $\Gamma^{*}$ with respect to $D^{*}$. For each $1\le j \le d$ we have
$$1\ll \lambda_j \lambda^*_{d-j+1}\ll 1.$$
\end{lemma} 

\section{Multiplicative energy of boxes in finite fields}
The following version of Siegel's Lemma is due to Bombieri and Vaaler~\cite{BV}.
\begin{lemma}
\label{lem:siegel}
Let $M$ and $L$ be integers with $M>L$. There exists a nontrivial integral solution $(t_1,\dots,t_M)$ to the system of equations
$$a_{\ell,1}t_1+\dots+a_{\ell,M}t_M=0 \quad \ell=1,\dots,L,$$
satisfying
$$\max_{1\le m \le M}|t_m|\le |\det{(A A^t)}|^{1/2(M-L)},$$
where $A$ denotes the matrix with $(\ell,m)$-th entry $a_{\ell,m}$ and $A^t$ denotes the transpose of $A$.
\end{lemma}

\begin{lemma}
\label{lem:minimab1}
Let $q$ be prime, $n$ an integer and $H_1,\dots,H_n$  integers satisfying
\begin{align}
\label{eq:Hcond1}
H_n\le H_{n-1}\le \dots \le H_1\le q,
\end{align}
and 
\begin{align}
\label{eq:Hcond2}
\prod_{k=1}^{i-1}H_k\ll \frac{qH_i^{i}}{H_{i-1}}, \quad 2\le i \le n.
\end{align}
Suppose $\omega_1,\dots,\omega_n$ is a basis for $\F_{q^{n}}$ as a vector space over $\F_{q}$. For $z\in \F_{q^{n}}$  let $\Gamma(z)$ denote the lattice 
\begin{align*}
& \Gamma(z)= \\ & \{ (x_1,\dots,x_n,y_1,\dots,y_n)\in \Z^{2n}  :  z(\omega_1 x_1+\dots+\omega_n x_n)=\omega_1 y_1+\dots+\omega_ny_n\},
\end{align*}
and $D$ the convex body
$$D=\{ (x_1,\dots,x_n,y_1,\dots,y_n)\in \R^{2n} : |x_i|,|y_i|\le H_i \ \ 1\le i \le n \}.$$
Let $\lambda_1(z),\dots,\lambda_{2n}(z)$ denote the successive minima of $\Gamma(z)$ with respect to $D$. For each $1\le i \le n$ we have
$$\lambda_i(z)\ge \frac{1}{H_i}.$$
\end{lemma}
\begin{proof}
We first note that 
$$\lambda_1(z)\ge \frac{1}{H_1},$$
as otherwise by~\eqref{eq:Hcond1}  
$$\lambda_1(z)D\cap \Z^{2n}=\{0\}.$$
Suppose for a contradiction that  for some $2\le i \le n$ we have
\begin{align}
\label{eq:cont1}
\lambda_i(z)< \frac{1}{H_i}.
\end{align}
 We may choose linearly independent points $$p_j=(x_{1,j},\dots,x_{n,j},y_{1,j},\dots,y_{n,j})\in \Gamma_z \cap \lambda_i(z) D, \quad j=1,\dots,i.$$
 By~\eqref{eq:cont1}, for each $1\le \ell \le n$ and $1\le j \le i$ we have 
$$|x_{\ell,j}|,|y_{\ell,j}|\le \lambda_i H_{\ell}<\frac{H_{\ell}}{H_i},$$
and hence by~\eqref{eq:Hcond1}
\begin{align*}
x_{\ell,j},y_{\ell,j}=0 \quad \text{for} \quad \ell \ge i.
\end{align*}
Projecting the points $p_1,\dots,p_{i}$ onto $2(i-1)$ dimensional space, we see that there exists linearly independent points 
\begin{align}
\label{eq:pj1}
p'_j=(x_{1,j},\dots,x_{i-1,j},y_{1,j},\dots,y_{i-1,j}) \in \Z^{2(i-1)}, \quad j=1,\dots,i,
\end{align}
such that
\begin{align}
\label{eq:xb1}
|x_{\ell,j}|,|y_{\ell,j}|\le  \frac{H_{\ell}}{H_i},
\end{align}
and
\begin{align}
\label{eq:pj2}
z(\omega_1x_{1,j}+\dots+\omega_{i-1}x_{i-1,j})=\omega_1 y_{1,j}+\dots+\omega_{i-1}y_{i-1,j}.
\end{align}
Consider the system of  equations
\begin{align}
\label{eq:linearsystem1}
t_1x_{\ell,1}+\dots+t_ix_{\ell,i}=0, \quad \ell=1,\dots,i-1,
\end{align}
in variables $t_1,\dots,t_i\in \Z$. Let $X$ denote the $(i-1) \times i$ matrix with $(k,j)$-th entry $x_{k,j}$ and $X^{t}$ denote the transpose of $X$. We see that the $(k,\ell)$-th entry of $XX^{t}$ is given by
$$\sum_{j=1}^{i}x_{k,j}x_{\ell,j}.$$
By~\eqref{eq:xb1} we have 
\begin{align*}
\sum_{j=1}^{i}x_{k,j}x_{\ell,j}\ll \frac{H_{k}H_{\ell}}{H^2_i},
\end{align*}
and hence by Hadamard's inequality
\begin{align*}
|\det{X}X^{T}|\ll \left(\frac{1}{H_i}\right)^{2(i-1)}\left(\prod_{k=1}^{i-1}H_k \right)^{2}.
\end{align*}
By Lemma~\ref{lem:siegel}, there exists an integral solution $t_1,\dots,t_i$ to~\eqref{eq:linearsystem1} such that
\begin{align}
\label{eq:tjbound}
|t_j|\ll \left(\frac{1}{H_i}\right)^{i-1}\prod_{k=1}^{i-1}H_k.
\end{align}
By~\eqref{eq:pj2} and~\eqref{eq:linearsystem1}  we have 
\begin{align*}
\omega_1 \sum_{j=1}^{i}t_j y_{1,j}+\dots+\omega_{i-1}\sum_{j=1}^{i}t_jy_{i-1,j}=0,
\end{align*}
and since $w_1,\dots,w_n$ are linearly independent over $\F_q$ 
\begin{align*}
\sum_{j=1}^{i}t_j y_{\ell,j}\equiv 0 \mod{q}, \quad \ell=1,\dots,i-1.
\end{align*}
By~\eqref{eq:xb1} and~\eqref{eq:tjbound} 
\begin{align*}
\sum_{j=1}^{i}t_j y_{\ell,j}\ll H_{i-1}\left(\frac{1}{H_i}\right)^{i}\prod_{k=1}^{i-1}H_k.
\end{align*}
which combined with~\eqref{eq:Hcond2} implies 
\begin{align*}
\sum_{j=1}^{i}t_j y_{\ell,j}=0, \quad \ell=1,\dots,i-1,
\end{align*}
contradicting the linear independence of the points~\eqref{eq:pj1}, so that
\begin{align*}
\lambda_i(z)\ge \frac{1}{H_i}.
\end{align*} 
\end{proof}
\begin{lemma}
\label{lem:gammastarminima}
Let $q$ be prime, $n$ an integer and  $H_1,\dots,H_n$  integers satisfying
\begin{align}
\label{eq:Hcond11}
H_n\le H_{n-1}\le \dots \le H_1\le q,
\end{align}
and 
\begin{align}
\label{eq:Hcond21}
H_{n-i+1}^{i}\ll qH_n\prod_{k=n-i+2}^{n}H_k, \quad 2\le i \le n,
\end{align}
for a sufficiently small implied constant. Suppose $\omega_1,\dots,\omega_n$ is a basis for $\F_{q^{n}}$ as a vector space over $\F_{q}$. For $z\in \F_{q^{n}}$  let $\Gamma(z)$ denote the lattice 
\begin{align*}
& \Gamma(z)= \\ & \{ (x_1,\dots,x_n,y_1,\dots,y_n)\in \Z^{2n}  :  z(\omega_1 x_1+\dots+\omega_n x_n)=\omega_1 y_1+\dots+\omega_ny_n\},
\end{align*}
and $D$ the convex body
$$D=\{ (x_1,\dots,x_n,y_1,\dots,y_n)\in \R^{2n} : |x_i|,|y_i|\le H_i \ \ 1\le i \le n \}.$$
Let $\lambda^{*}_1(z),\dots,\lambda^{*}_{2n}(z)$ denote the successive minima of $\Gamma^{*}(z)$ with respect to $D^{*}$, where $\Gamma^{*}(z)$ and $D^{*}$ are the dual lattice and dual body. Then for each $1\le i \le n$ we have 
$$\lambda^{*}_i(z)\gg \frac{H_{n-i+1}}{q}.$$
\end{lemma}
\begin{proof}
We first note that the dual lattice $\Gamma^{*}(z)$ and dual body $D^{*}$ are given by 
\begin{align*}
&\Gamma^{*}(z)= \\ & \left\{ \left(\frac{t}{q},\frac{s}{q} \right), (s,t)\in \Z^{2n}  :  \sum_{i=1}^{n}t_ix_i+\sum_{i=1}^{n}s_iy_i\equiv 0 \mod{q} \ \text{for all $(x,y)\in \Gamma(\omega)$}  \right\},
\end{align*}
and 
\begin{align*}
D^{*}=\left\{ (t,s)\in \R^{2n} \ : \ \sum_{i=1}^{n}H_i|t_i|+\sum_{i=1}^{n}H_i|s_i|\le 1 \right\}.
\end{align*}
We see that 
$$\lambda_1^{*}(z)\gg \frac{H_n}{q},$$
since 
\begin{align*}
\Gamma^{*}(z)\cap \frac{\varepsilon H_n}{q} D^{*}=\{0\},
\end{align*}
for a sufficiently small $\varepsilon$ depending only on $n$.  Let $2\le i \le n$ and suppose for a contradiction that 
\begin{align}
\label{eq:lambdastartassumption}
\lambda^{*}_i(z)\ll  \frac{ H_{n-i+1}}{q},
\end{align}
for a sufficiently small implied constant depending only on $n$.  By~\eqref{eq:lambdastartassumption} there exists linearly independent points 
$$p_j=(t_{1,j},\dots,t_{n,j},s_{1,j},\dots,s_{n,j})\in \varepsilon H_{n-i+1}D^{*}\cap \Z^{2n}, \quad 1\le j\le i,$$
such that for each $1\le j \le i$ we have
\begin{align*}
|t_{\ell,j}|, |s_{\ell,j}|\ll \frac{\varepsilon H_{n-i+1}}{H_{\ell}},
\end{align*}
and hence by~\eqref{eq:Hcond11} we have 
\begin{align*}
t_{\ell,j},s_{\ell,j}=0 \quad \text{for} \quad \ell\le  n-i+1.
\end{align*}
Projecting the $p_j$ onto $2(i-1)$ dimensional space, there exists linearly independent points  
\begin{align}
\label{eq:pjstar}
p_j'=(t_{n-i+2,j},\dots,t_{n,j},s_{n-i+2,j},\dots,s_{n,j})\in \Z^{2(i-1)} \quad 1\le j \le i,
\end{align}
satisfying
\begin{align}
\label{eq:tsbounds1}
|t_{\ell,j}|, |s_{\ell,j}|\ll \frac{ H_{n-i+1}}{H_{\ell}},
\end{align}
and for each $1\le j \le i$
\begin{align}
\label{eq:gammastarequiv}
t_{n-i+2,j}x_{n-i+2}+\dots+t_{n,j}x_{n}+s_{n-i+2,j}y_{n-i+2}+\dots+s_{n,j}y_{n}\equiv 0 \mod{q}.
\end{align}
for every $(x_1,\dots,y_n)\in \Gamma(z)$. Consider the system of equations 
\begin{align*}
b_1t_{m,1}+\dots+b_it_{m,i}=0, \quad n-i+2\le m \le n.
\end{align*}
By Lemma~\ref{lem:siegel}, there exists a nontrivial integral solution $b_1,\dots,b_i$ satisfying
\begin{align}
\label{eq:bjbound}
|b_j|\ll H_{n-i+1}^{i-1}\prod_{k=n-i+2}^{n}\frac{1}{H_k}, \quad 1\le j \le i,
\end{align}
and hence by~\eqref{eq:gammastarequiv}
\begin{align}
\label{eq:start1111111111}
\left(\sum_{j=1}^{i}b_js_{n-i+2,j}\right)y_{n-i+2}+\dots+\left(\sum_{j=1}^{i}b_js_{n,j}\right)y_{n}\equiv 0 \mod{q},
\end{align}
for every tuple $(y_{n-i+2},\dots,y_n)$ such that there exists $x_1,\dots,x_n,y_1,\dots y_{n-i+1}$ with  $(x_1,\dots,y_n)\in \Gamma(z).$ Since $\omega_1,\dots,\omega_n$ forms a basis for  $\F_{q^{n}}$ as a vector space over $\F_q$, for an arbitrary choice of $y_{n-i+2},\dots,y_n\in \F_q$ there exists $x_1,\dots,x_n\in \F_q$ such that 
$$z(\omega_1x_1+\dots+\omega_nx_n)=\omega_{n-i+2}y_{n-i+2}+\dots+\omega_n y_n,$$
and hence by~\eqref{eq:start1111111111} we have
\begin{align*}
\sum_{j=1}^{i}b_js_{\ell,j}\equiv 0 \mod{q}, \quad n-i+2\le \ell \le n.
\end{align*}
By~\eqref{eq:Hcond21},~\eqref{eq:tsbounds1} and~\eqref{eq:bjbound}
\begin{align*}
\sum_{j=1}^{i}b_js_{\ell,j}\ll \frac{H_{n-i+1}^{i}}{H_n}\prod_{k=n-i+2}^{n}\frac{1}{H_k}<q,
\end{align*}
and hence 
\begin{align*}
\sum_{j=1}^{i}b_js_{\ell,j}=0, \quad n-i+2\le \ell \le n,
\end{align*}
contradicting the the fact that the points~\eqref{eq:pjstar} are linearly independent. This gives
\begin{align*}
\lambda^{*}_i(z)\gg  \frac{\varepsilon H_{n-i+1}}{q}.
\end{align*}
\end{proof}

\section{Proof of Theorem~\ref{thm:main1}}
For $z\in \F_{q^{n}}$ we let $I(z)$ count the number of solutions to the equation
\begin{align}
\label{eq:Izdef}
z(\omega_1x_1+\dots+\omega_nx_n)=\omega_1 y_1+\dots +\omega_ny_n, \quad M_i\le x_i,y_i\le M_i+H_i,
\end{align}
so that 
\begin{align*}
E(B)=\sum_{z \in \F_{q^n}}I(z)^2\le \sum_{\substack{z \in \F_{q^n} \\  z\neq 0 }}I(z)^2+|B|^2.
\end{align*}
We define the lattice 
$$\Gamma(z)=\{ (x_1,\dots,x_n,y_1,\dots,y_n)\in \Z^{2n}  :  z(\omega_1 x_1+\dots+\omega_n x_n)=\omega_1 y_1+\dots+\omega_ny_n\},$$
and the convex body
$$D=\{ (x_1,\dots,x_n,y_1,\dots,y_n)\in \R^{2n} : |x_i|,|y_i|\le H_i \ \ 1\le i \le n \}.$$
 For any two points $(x_1,\dots,y_n)$ and $(x_1',\dots,y_n')$ satisfying~\eqref{eq:Izdef} we have 
\begin{align*}
(x_1-x_1',\dots,y_n-y_n')\in \Gamma(z)\cap D,
\end{align*}
and hence
\begin{align*}
E(B)\le  \sum_{\substack{z \in \F_{q^n} \\  z\neq 0 }}|\Gamma(z)\cap D|^2+|B|^2.
\end{align*}
Let 
$$\Omega'=\{ z\in \F_{q^{n}}/\{0\} \ : \  \Gamma(z)\cap D\neq \{ 0\} \},$$
so that 
\begin{align}
\label{eq:Izlattice123}
E(B)\ll  \sum_{\substack{z \in \Omega'}}|\Gamma(z)\cap D|^2+|B|^2.
\end{align}
Let $\lambda_1(z),\dots,\lambda_{2n}(z)$ denote the successive minima of $\Gamma(z)$ with respect to $D$ and define
$$s(z)=\max\{ j  :  \lambda_j(z)\le 1 \}.$$
If $z\in \Omega'$ then $s(z)\ge 1$ and hence we may partition
\begin{align}
\label{eq:GammaDS}
\sum_{\substack{z \in \Omega'}}|\Gamma(z)\cap D|^2=\sum_{j=1}^{2n}S_j,
\end{align}` 
where 
$$S_j=\sum_{\substack{z \in \Omega_j}}|\Gamma(z)\cap D|^2.$$
Fix some $1\le j \le 2n$ and consider $S_j$. We first suppose that $1\le j \le n$. By Lemma~\ref{lem:latticesm}
\begin{align*}
S_j\ll \sum_{z\in \Omega_j}\prod_{i=1}^{j}\frac{1}{\lambda_i(z)^2}.
\end{align*}
For a $j$-tuple of integers  $(k_1,\dots,k_j)$ let
\begin{align*}
\Omega_j(k_1,\dots,k_j)=\{ z\in \Omega_j \ : 2^{-k_i-1}< \lambda_i(z)\le 2^{-k_i}, \ \ i=1,\dots,k \}.
\end{align*}
Since $\lambda_1(z)\le \lambda_2(z)\le \dots \le \lambda_j(z)$ we must have 
$$2^{k_i-1}\le 2^{k_1}, \quad 1\le i \le j,$$
and  by~\eqref{eq:Hcondthm1} and Lemma~\ref{lem:minimab1} 
\begin{align*}
\Omega(k_1,\dots,k_j)=\emptyset \quad \text{unless $2^{k_i}\le H_i$ for each $1\le i \le j$},
\end{align*}
which gives 
\begin{align*}
S_j \ll \sum_{\substack{k_1,\dots,k_j\ge 0 \\ 2^{k_i}\le H_i \\ 2^{k_i}\le 2^{k_1+1}}}2^{2(k_1+\dots+k_j)}|\Omega(k_1,\dots,k_j)|.
\end{align*}
Considering $\Omega(k_1,\dots,k_j)$, since each point $(x_1,\dots,y_n)\in D\cap \Z^{2n}$ can belong to at most one lattice $\Gamma(z)$ we have  
\begin{align*}
|\Omega(k_1,\dots,k_j)|\ll \left|2^{-k_1+1}D\cap \Z^{2n}\right|\ll \prod_{i=1}^{n}\left(\frac{H_i}{2^{k_1}}+1 \right)^2,
\end{align*}
 and hence 
\begin{align}
\label{eq:Sjbound}
S_j\ll \sum_{\substack{k_1,\dots,k_j \\ 2^{k_i}\le H_i \\ 2^{k_i}\le 2^{k_1+1}}}\prod_{\ell=1}^{n}\left(\frac{2^{k_{i}}H_i}{2^{k_1}}+2^{k_{i}} \right)^2\ll \prod_{i=1}^{n}H_i^2\sum_{\substack{k_1,\dots,k_j \\ 2^{k_i}\le H_i \\ 2^{k_i}\le 2^{k_1+1}}}1\ll |B|^2(\log{|B|})^{n},
\end{align}
where we set $k_{i}=0$ in the above sum if $i>j$.

 Consider next estimating $S_j$ when $n+1\le j \le 2n$. 
 If $z\in \Omega_j$ then by Lemma~\ref{lem:latticesm} and Lemma~\ref{lem:transfer}
\begin{align*}
|\Gamma(z)\cap D|\ll \prod_{i=1}^{j}\frac{1}{\lambda_i(z)}=\prod_{i=1}^{2n}\frac{1}{\lambda_i(z)}\prod_{i=j+1}^{2n}\lambda_i(z)\ll \prod_{i=1}^{2n}\frac{1}{\lambda_i(z)}\prod_{i=1}^{2n-j}\frac{1}{\lambda^{*}_i(z)},
\end{align*}
where $\lambda^{*}_i(z)$ denote the successive minima of the dual lattice $\Gamma^{*}(z)$ with respect to the dual body $D^{*}$. By Lemma~\ref{lem:mst}
\begin{align*}
\prod_{i=1}^{2n}\frac{1}{\lambda_i(z)}\ll \frac{(H_1\dots H_n)^2}{q^n},
\end{align*}
so that 
\begin{align*}
S_j\ll \frac{(H_1\dots H_n)^4}{q^{2n}}\sum_{\lambda \in \Omega_j}\prod_{i=1}^{2n-j}\frac{1}{\lambda^{*}_i(z)^2}.
\end{align*}
We have 
\begin{equation}
\label{eq:S2nbound}
S_{2n}\ll \frac{(H_1\dots H_n)^4}{q^{2n}}\sum_{z\in \F_{q^{n}}}1= \frac{|B|^4}{q^{n}},
\end{equation}
and it remains to consider when $n+1\le j \le 2n-1$. Writing 
$$j=2n-\ell,$$
for some $1\le \ell \le n-1$, we have 
\begin{align}
\label{eq:S2nell}
S_{2n-\ell}\ll \frac{(H_1\dots H_n)^4}{q^{2n}}T_{\ell},
\end{align}
where 
$$T_{\ell}=\sum_{\lambda \in \Omega_{2n-\ell}}\prod_{i=1}^{\ell}\frac{1}{\lambda^{*}_i(z)^2}.$$
For an $\ell$-tuple of integers $(k_1,\dots,k_{\ell})$  define
$$\Omega(k_1,\dots,k_{\ell})=\{ \lambda \in \Omega_{2n-\ell} \ : 2^{-k_i-1}<\lambda_i^{*}\le 2^{-k_i}, \ i=1,\dots,\ell \}.$$
Since $\lambda_1^{*}(z)\le \lambda_2^{*}(z) \le \dots \le \lambda_{\ell}^{*}(z)$ we must have 
\begin{align*}
2^{k_i-1}\le 2^{k_1},
\end{align*} 
and by~\eqref{eq:Hcondthm2} and Lemma~\ref{lem:gammastarminima}
\begin{align*}
\Omega(k_1,\dots,k_{\ell})=\emptyset \quad \text{unless $2^{k_i}\ll \frac{q}{H_{n-i+1}}$ for each $1\le i \le \ell$}.
\end{align*}
Since the contribution to $T_{\ell}$ from those $\lambda\in \Omega_{2n-\ell}$ with $\lambda_1^{*}(z)\ge 1$ is $O(q^{n})$ we see that
\begin{align*}
T_{\ell}\ll \sum_{\substack{k_1,\dots,k_{\ell}\ge 0 \\ 2^{k_i}\ll q/H_{n-i+1} \\ 2^{k_i}\le 2^{k_1+1}}}2^{2(k_1+\dots+k_{\ell})}|\Omega(k_1,\dots,k_\ell)|+q^{n}.
\end{align*}
Proceeding as in~\cite{Kon}, we next show that each point $(s_1/q,\dots,t_n/q)\in D^{*}\cap\Z^{2n}/q$ can belong to at most one lattice $\Gamma^{*}(z).$ If this were false then there would exist a tuple of integers $(t,s)$ and $z,z'\in \F_{q^{n}}$ such that $(t/q,s/q)\in \Gamma^{*}(z_1)\cap D^{*}$ and $(t/q,s/q)\in \Gamma^{*}(z_2)\cap D^{*}$. Since $\omega_1,\dots,\omega_n$ form a basis for $\F_{q^{n}}$ over $\F_q$, for every $x_1,\dots,x_n\in \F_q$ there exists $y_1,\dots,y_n\in \F_q$ and $y_1',\dots,y_n'\in\F_q$ such that 
\begin{align*}
(x_1,\dots,y_n)\in \Gamma(z), \quad (x_1,\dots,y_n')\in \Gamma(z'),
\end{align*}
which implies 
\begin{align*}
(y_1-y_1')t_1+\dots+(y_n-y_n')t_n\equiv 0 \mod{q}.
\end{align*}
We see that
$x(z-z')=y-y',$
where 
$$x=\omega_1x_1+\dots+\omega_nx_n, \quad y=\omega_1y_1+\dots+\omega_ny_n, \quad y'=\omega_1y'_1+\dots+\omega_ny'_n,$$
and hence we may choose $x$ so that $y-y'$ takes an arbitrary value in $\F_{q^n}$. This implies that 
\begin{align*}
t_i\equiv 0 \mod{q}, \quad 1\le i \le n,
\end{align*}
and since $(s/q,t/q)\in D^{*}$ we must have 
\begin{align*}
t_i=0, \quad 1\le i \le n.
\end{align*}
In a similar fashion we may show $s_i=s_i'$. Since each point $(s_1/q,\dots,t_n/q)\in D^{*}\cap\Z^{2n}/q$ can belong to at most one lattice $\Gamma^{*}(z)$ we have 
\begin{align*}
|\Omega(k_1,\dots,k_\ell)|\ll \left|2^{-k_1+1}D^{*}\cap\frac{\Z^{2n}}{q}\right|\ll \prod_{i=1}^{n}\left(\frac{q}{2^{k_1}H_{n-i+1}}+1 \right)^2,
\end{align*}
and hence 
\begin{align*}
T_{\ell}\ll \sum_{\substack{k_1,\dots,k_{\ell}\ge 0 \\ 2^{k_i}\ll q/H_{n-i+1} \\ 2^{k_i}\le 2^{k_1+1}}}\prod_{i=1}^{n}\left(\frac{2^{k_i}q}{2^{k_1}H_{n-i+1}}+2^{k_i} \right)^2\ll \frac{q^{2n}}{(H_1\dots H_n)^2}(\log{H_1})^{n},
\end{align*}
where we set $k_i=0$ in the above sum if $i>\ell$. Combining the above with~\eqref{eq:S2nell} we get 
\begin{align*}
S_{2n-\ell}\ll |B|^2(\log{|B|})^{n},
\end{align*}
and hence by~\eqref{eq:Izlattice123},~\eqref{eq:GammaDS},~\eqref{eq:Sjbound} and~\eqref{eq:S2nbound}
\begin{align*}
E(B)\ll \frac{|B|^{4}}{q^n}+|B|^2(\log{|B|})^{n},
\end{align*}
which completes the proof.
\section{Multiplicative energy of generalized arithmetic progressions}
For two  $d$-tuples of real numbers $\varepsilon=(\varepsilon_1,\dots,\varepsilon_d)$ and $\alpha=(\alpha_1,\dots,\alpha_d)$ we define the Bohr set
\begin{equation}
\label{eq:bohrdef}
B(\alpha,\varepsilon)=\left\{ 1\le x \le q-1 : \left\|\frac{\alpha_i x}{q} \right\|\le \varepsilon_i \quad i=1,\dots,d \right\},
\end{equation}
and for a  generalized arithmetic progression $\cA$ given by
$$\cA=\{ \alpha_1h_1+\dots+\alpha_dh_d \ : \ 1\le h_i \le H\},$$
 we let $E(\cA,\varepsilon)$ count the number of solutions to the congruence
$$a_1b_1\equiv a_2b_2 \mod q \quad a_1,a_2 \in \cA, \quad b_1,b_2\in B(\alpha, \varepsilon).$$

The following is based on some ideas of Ayyad, Cochrane and Zheng~\cite{ACZ}.
\begin{lemma}
\label{lem:reduction}
 With notation as above, suppose that $\cA$ is proper. Then we have 
$$E(\cA)=\frac{|\cA|^4}{q}+O\left(\frac{(\log{H})^{2d}}{q}\max_{\substack{1/H\le \varepsilon_i \le 1}}\frac{E(\alpha,\varepsilon)}{(\varepsilon_1\dots\varepsilon_d)^2} \right).$$
\end{lemma}
\begin{proof}
Let $\cA(x)$ denote the indicator function of the set $\mathcal{A}$ and let $\widehat \cA(y)$ denote the Fourier coefficients of $\cA(x)$, so that
\begin{align}
\label{eq:afc}
|\widehat \cA(y)|\ll \frac{1}{q}\prod_{i=1}^{d}\left(H,\frac{1}{\|\alpha_i y/q\|}\right) \quad \text{and} \quad \widehat \cA(0)=\frac{|\cA|}{q}.
\end{align}
Since $\cA$ is proper $0\not \in \cA$ and hence 
\begin{align*}
E(\cA,\cA)&=\sum_{\substack{ a_1,a_2,a_3\in \mathcal{A} \\ }} \cA\left(a_1a_2a_3^{-1} \right) \\
&=\frac{|\cA|^4}{q}+\sum_{y=1}^{q-1}\widehat \cA(y)\sum_{ a_1,a_2,a_3\in \mathcal{A} }e_q\left(a_1a_2a_3^{-1}y \right) \\
&=\frac{|\cA|^4}{q}+\sum_{y=1}^{q-1}\sum_{z=1}^{q-1}\widehat \cA(y) \overline{\widehat \cA(z)}\sum_{\substack{a_2,a_3\in \cA \\ a_2y\equiv a_3z \mod q}}1,
\end{align*}
which combined with~\eqref{eq:afc} implies
\begin{align}
\label{eq:errorub}
& E(\cA,\cA)-\frac{|\cA|^{4}}{q}\ll \nonumber \\ & \quad  \sum_{y=1}^{q-1}\sum_{z=1}^{q-1}\prod_{i=1}^{d}\left(H,\frac{1}{\|\alpha_i y/q\|}\right)\prod_{j=1}^{d}\left(H,\frac{1}{\|\alpha_j z/q\|}\right)\sum_{\substack{a_2,a_3\in \cA \\ a_2y \equiv a_3z \mod q}}1.
\end{align}
For a $d$-tuple of integers $j=(j_1,\dots j_d)$  we define the sets
\begin{align*}
& \quad  B(j)= \\ & \left \{ 1\le y \le q-1 \   :  \  \frac{(2^{j_i}-1)}{H}\le \left\|\frac{\alpha_i y}{q}\right\|<\frac{(2^{j_i+1}-1)q}{H}, \ \ \ 1\le i \le d \right \},
\end{align*}
so that as each $j_i$ ranges over values $0\le j_i \ll \log{H}$ the sets $B(j)$ cover the interval $1\le y \le q-1$ and if $y\in B(j)$ we have
\begin{align*}
\prod_{i=1}^{d}\left(H,\frac{1}{\|\alpha_i y/q\|}\right)\ll \prod_{i=1}^{d}\frac{H}{2^{j_i}},
\end{align*}
which gives
\begin{align*}
&\sum_{y=1}^{q-1}\sum_{z=1}^{q-1}\prod_{i=1}^{d}\left(H,\frac{1}{\|\alpha_i y/q\|}\right)\prod_{j=1}^{d}\left(H,\frac{1}{\|\alpha_j z/q\|}\right)\sum_{\substack{a_2,a_3\in \cA \\ a_2y \equiv a_3z \mod q}}1 \\
&\ll \sum_{\substack{j \\ 0\le j_i \ll \log{H}}}\sum_{\substack{k \\ 0\le k_i \ll \log{H}}}\prod_{i=1}^{d}\frac{H^2}{2^{j_i+k_i}}\sum_{y\in B(j)}\sum_{z\in B(k)}\sum_{\substack{a_2,a_3\in \cA \\ a_2y \equiv a_3z \mod q}}1.
\end{align*}
For $1\le i \le d$ let 
\begin{align*}
\varepsilon_{j,i}=2^{j_i+1}/H, \quad \varepsilon_{k,i}=2^{k_i+1}/H,
\end{align*}
and write
$$\varepsilon_j=(\varepsilon_{j,1},\dots,\varepsilon_{j,d}), \quad \varepsilon_k=(\varepsilon_{k,1},\dots,\varepsilon_{k,d}).$$ With $B(\alpha,\varepsilon)$ given as in~\eqref{eq:bohrdef}, we have
\begin{align*}
\sum_{y\in B(j)}\sum_{z\in B(k)}\sum_{\substack{a_2,a_3\in \cA \\ a_2y \equiv a_3z \mod q}}1&\le
\sum_{y\in B(\alpha,\varepsilon_j)}\sum_{z\in B(\alpha,\varepsilon_k)}\sum_{\substack{a_2,a_3\in \cA \\ a_2y \equiv a_3z \mod q}}1 \\
&=\sum_{w=1}^{q-1}\left(\sum_{\substack{y\in B(\alpha,\varepsilon_j)  \ a_3 \in \cA \\ ya_3^{-1}\equiv w \mod q }}1 \right)\left(\sum_{\substack{z\in B(\alpha,\varepsilon_k) \ a_1 \in \cA \\ za_1^{-1}\equiv w \mod q }}1 \right).
\end{align*}
Since
$$E(\cA,\varepsilon)=\sum_{w=1}^{q-1}\left(\sum_{\substack{y\in B(\alpha,\varepsilon) \ a_3 \in \cA \\ ya_3^{-1}\equiv w \mod q }}1\right)^2,$$
 an application of the  Cauchy-Schwarz inequality gives
\begin{align*}
\sum_{y\in B(j)}\sum_{z\in B(k)}\sum_{\substack{a_2,a_3\in \cA \\ a_2y \equiv a_3z \mod q}}1\le E(\cA,\varepsilon_j)^{1/2}E(\cA,\varepsilon_k)^{1/2}.
 \end{align*}
Substituting the above into~\eqref{eq:errorub} we arrive at 
\begin{align*}
E(\cA,\cA)-\frac{|\cA|^4}{q}&\ll \sum_{\substack{j \\ 0\le j_i \ll \log{H}}}\frac{1}{\varepsilon_{j,1}\dots \varepsilon_{j,d}}E(\cA,\varepsilon_j)^{1/2} \\
& \times \sum_{\substack{k \\ 0\le k_i \ll \log{H}}}\frac{1}{\varepsilon_{k,1}\dots \varepsilon_{k,d}}E(\cA,\varepsilon_j)^{1/2},
\end{align*}
and the result follows since there are $O((\log{H})^{2d})$ terms in summation over $j$ and $k$.
\end{proof}
The following is due to Shao~\cite[Proposition 2.1]{Shao1}.
 \begin{lemma}
\label{lem:shao}
For integers $q$ and $H$  and  a $d$-tuple of integers $\alpha=(\alpha_1,\dots,\alpha_d)$  suppose that the equation 
$$\alpha_1h_1+\dots+\alpha_dh_d\equiv 0 \mod q,$$
has no nontrivial solutions in integers $|h_i|\le H.$ Then for $\varepsilon=(\varepsilon_1,\dots,\varepsilon_d)$  with each $0\le \varepsilon_i\le 1/2$ the cardinality of the Bohr set
$$B(\alpha,\varepsilon)=\left\{ 1\le x \le q : \left\|\frac{\alpha_i x}{q}\right\|\le \varepsilon_i \quad i=1,\dots,d \right\},$$  
satisfies
\begin{align*}
|B(\alpha,\varepsilon)|\ll q \prod_{i=1}^{r}\left(\varepsilon_i+\frac{1}{H_i}\right).
\end{align*}
\end{lemma}


\section{Proof of Theorem~\ref{thm:main2}}
We first note the assumption
\begin{align}
\label{eq:A'proper}
\cA'=\{ \alpha_1h_1+\dots+\alpha_dh_d \ : \ |h_i| \le H^2 \},
\end{align}
is proper implies that 
\begin{align*}
H\le q^{1/2d},
\end{align*}
and in particular 
$$\frac{H^{4d}}{q}\le H^{2d}.$$
Hence by Lemma~\ref{lem:reduction} it is sufficient to show that 
\begin{align}
\label{eq:thm2s1}
\max_{\substack{1/H\le \varepsilon_i \le 1}}\frac{E(\alpha,\varepsilon)}{(\varepsilon_1\dots\varepsilon_d)^2}\ll qH^{2d}(\log{H}).
\end{align}
Suppose 
\begin{align}
\label{eq:varepsilondeltadef}
\varepsilon=\left(\frac{\delta_1}{H},\dots,\frac{\delta_d}{H} \right),
\end{align}
is such that the expression occuring in~\eqref{eq:thm2s1} is maximum for some $\delta_1,\dots,\delta_d\ge 1$. We have
\begin{align}
\label{eq:Eell2}
E(\cA,\varepsilon)=\sum_{\omega=1}^{q-1}I(\omega)^2,
\end{align}
where $I(\omega)$ counts the number of solutions to the congruence
\begin{align}
\label{eq:Idef}
a\equiv \omega b \mod q \quad a\in \cA, \quad b\in B(\alpha,\varepsilon).
\end{align}
We define 
$$\cL=\{ (y_1,\dots,y_d)\in \Z^{d} : \exists \  1\le x \le q \ \ \text{such that} \ \   y_i\equiv \alpha_ix \mod q \},$$
and for each $1\le \omega \le q-1$ let $\Gamma(\omega)$ denote the lattice
$$\Gamma(\omega)=\{ (h,y)\in \Z^{d}\times \cL : \langle \alpha, h\rangle \equiv d^{-1}\omega \langle \alpha^{-1},  y \rangle \mod q \},$$
where $\langle , \rangle$ denotes the Euclidian inner product and $\alpha^{-1}$ denotes the vector formed by taking the inverse mod $q$ of each coordinate of $\alpha$, so that
$$\alpha^{-1}=(\alpha_1^{-1},\dots,\alpha_d^{-1}).$$ 
 Let $D(\delta)$ denote the convex body
$$D(\delta)=\left \{ (t,s)\in \R^{d}\times \R^{d} : |t_j|\le H, \quad |s_j|\le \frac{\delta_i q}{H} \right \}.$$
Since $\cA$ is proper, the set of points  $(h,y)\in \Gamma(w)\cap D(\delta)$ with $h\in \Z^{d}$, $y\in \cL$ and $(h,y)\neq 0$ is  in one-to-one correspondence with 
solutions to the congruence~\eqref{eq:Idef} via 
$$(h,y)\rightarrow (\alpha_1h_1+\dots+\alpha_dh_d,b),$$
where $b$ is defined by $y_i\equiv \alpha_i b.$ By~\eqref{eq:Eell2} this implies 
$$E(\cA,\varepsilon)\le \sum_{\omega=1}^{q-1}|\Gamma(\omega)\cap D(\varepsilon)|^2,$$
and hence by~\eqref{eq:thm2s1} and~\eqref{eq:varepsilondeltadef} it is sufficient to show that 
\begin{equation}
\label{eq:thm2s2}
\sum_{w=1}^{q-1}|\Gamma(w)\cap D(\varepsilon)|^2\ll q(\delta_1\dots \delta_d)^2(\log{H}).
\end{equation}

Let 
$$\Omega'=\{ 1\le \omega \le q-1 : \Gamma(w)\cap D(\varepsilon)\neq \{ 0 \} \},$$
so that
\begin{equation}
\label{eq:uplatticesum}
\sum_{w=1}^{q-1}|\Gamma(w)\cap D(\varepsilon)|^2\ll \sum_{\omega \in \Omega'}|\Gamma(w)\cap D(\varepsilon)|^2+  q.
\end{equation}
For each $1\le \omega\le q-1$ we let $\lambda_1(\omega),\dots,\lambda_{2d}(\omega)$ denote the successive minima of $D(\varepsilon)$ with respect to $\Gamma(\omega)$ and let  $\lambda^*_1(w),\dots,\lambda^*_{2d}(\omega)$ denote  the successive minima of $D^*(\varepsilon)$ with respect to $\Gamma^*(\omega)$. Considering points $(h,y)\in \Gamma(\omega)$, each $h\in \Z^d$ uniquley determines the residue mod $q$ of each coordinate of $y \mod q$ so that
$$\text{Vol}(\R^{2d}/\Gamma(\omega))=q^{d},$$
and since $|D|=q^{d}\delta_1\dots\delta_d,$ an application of Lemma~\ref{lem:mst} gives 
\begin{equation}
\label{eq:mst}
 \lambda_1(\omega)\dots\lambda_{2d}(\omega)\gg \frac{1}{\delta_1\dots\delta_d}.
\end{equation}
For each $\omega\in \Omega'$ we define the integer $s(\omega)$ by $$s(\omega)=\max\{ j : \lambda_j(\omega)\le 1\},$$
so that 
$$s(\omega)\ge 1 \quad \text{if} \quad \omega\in \Omega.$$
Let
\begin{align*}
\Omega=\{ \omega\in \Omega' \ : \ 1\le s(\omega)\le d\}, \quad \Omega^{*}=\{ \omega\in \Omega' \ : \ d+1\le s(\omega)\le 2d\},
\end{align*} 
and write 
\begin{align}
\label{eq:OmegaSS*}
\sum_{\omega \in \Omega'}|\Gamma(w)\cap D(\varepsilon)|^2=S+S^{*},
\end{align}
where
\begin{align*}
S=\sum_{\omega\in \Omega}|\Gamma(w)\cap D(\varepsilon)|^2,
\end{align*}
and 
\begin{align*}
S^{*}=\sum_{\omega\in \Omega^{*}}|\Gamma(w)\cap D(\varepsilon)|^2.
\end{align*}
Considering $S$, we partition $\Omega$ into
$$\Omega_j=\{ \omega \in \Omega \ : \ s(\omega)=j\},$$
and write 
\begin{align}
\label{eq:SSj}
S=\sum_{j=1}^{d}S_j,
\end{align}
where 
$$S_j=\sum_{\omega\in \Omega_j}|\Gamma(w)\cap D(\varepsilon)|^2.$$
If $\omega \in \Omega_j$ then by Lemma~\ref{lem:latticesm} we have 
\begin{align*}
|\Gamma(\omega)\cap D(\varepsilon)|\ll \prod_{i=1}^{j}\frac{1}{\lambda_i(\omega)},
\end{align*}
and hence  
\begin{align}
\label{eq:Sjb1}
S_j\ll \sum_{\omega\in \Omega_j}\prod_{i=1}^{j}\frac{1}{\lambda_i(\omega)^2}\ll \sum_{\omega\in \Omega_j}\frac{1}{\lambda_1(\omega)^{2j}}.
\end{align}
For integer $k$ we define the set
\begin{align*}
\Omega_j(k)=\{ \omega \in \Omega_j \ : \  2^{-(k+1)}\le \lambda_1(\omega)< 2^{-k} \},
\end{align*}
so that 
\begin{align}
\label{eq:omegaempty}
\Omega_j(k)=\emptyset \quad \text{if $2^{k}>H$},
\end{align}
and by~\eqref{eq:Sjb1}
\begin{align}
\label{eq:Sj1111111111}
S_j\ll \sum_{\substack{k \\ 2^{k}\le H}}2^{2jk}|\Omega_j(k)|.
\end{align}
Since each nonzero point $(h,y)\in \Z^d\times \cL$ can belong to at most one lattice $\Gamma(\omega)$, we see that
\begin{align}
\label{eq:omegab1111111111}
|\Omega_j(k)|\ll \frac{H^d}{2^{kd}}\left|B\left(\alpha,\frac{\delta}{2^{k}H}\right)\right|,
\end{align}
where
$$B(\alpha,\varepsilon)=\left\{ 1\le x \le q : \left\|\frac{\alpha_i x}{q}\right\|\le \frac{\delta_i}{2^kH} \quad i=1,\dots,d \right\}.$$  
By~\eqref{eq:A'proper}, ~\eqref{eq:omegaempty} and Lemma~\ref{lem:shao}
\begin{align*}
|B(\alpha,\varepsilon)|\ll q \prod_{i=1}^{r}\left(\frac{\delta_i}{2^kH}+\frac{1}{H^2}\right)\ll \frac{q(\delta_1\dots \delta_d)}{2^{kd}H^d},
\end{align*}
which combined with~\eqref{eq:Sj1111111111} and~\eqref{eq:omegab1111111111} gives 
\begin{align*}
S_j\ll q(\delta_1 \dots \delta_d)\sum_{\substack{k \\ 2^{k}\le H}}\frac{1}{2^{2(d-j)k}}\ll q(\delta_1 \dots \delta_d)(\log{H}),
\end{align*}
and hence by~\eqref{eq:SSj}
\begin{align}
\label{eq:Sb123}
S\ll q(\delta_1 \dots \delta_d)(\log{H}).
\end{align}

Considering $S^{*}$, we first note that  dual lattice $\Gamma^{*}(\omega)$ and dual body $D^{*}(\delta)$ are given by 
\begin{align*}
\Gamma^{*}(\omega)=\left\{ \left(\frac{h}{q},\frac{y}{q}\right), \ h\in \cL, \ y\in \Z^{d} \ : \langle \alpha,y \rangle \equiv -d\omega \langle \alpha^{-1},h\rangle  \right\},
\end{align*}
and
\begin{align*}
D^{*}(\delta)=\left\{ (t,s)\in \R^{d}\times \R^{d} \ : \sum_{i=1}^{d}|t_i|H+\sum_{i=1}^{d}\frac{\delta_i q}{H}|s_i|\le 1 \right\}.
\end{align*}
For integer $1\le j \le d$ we let
$$\Omega^{*}_j=\{ \omega \in \Omega^{*} \ : \ s(\omega)=d+j\},$$
and partition 
\begin{align}
\label{eq:SSj}
S^{*}=\sum_{j=1}^{d}S^{*}_j,
\end{align}
where 
$$S^{*}_j=\sum_{\omega\in \Omega^{*}_j}|\Gamma(w)\cap D(\varepsilon)|^2.$$

 Fix some $1\le j \le d$ and consider $S_j^{*}$. If  $\omega\in \Omega^{*}(j)$ then by Lemma~\ref{lem:latticesm} and~\eqref{eq:mst} we have 
\begin{align*}
|\Gamma(w)\cap D(\varepsilon)|^2&\ll \prod_{i=1}^{d+j}\frac{1}{\lambda_i(\omega)^2} \\ &=\prod_{i=1}^{2d}\frac{1}{\lambda_i(\omega)^2}\left(\prod_{i=d+j+1}^{2d}\lambda_i(\omega)^2\right) \\ & \ll (\delta_1 \dots \delta_d)^2\left(\prod_{i=d+j+1}^{2d}\lambda_i(\omega)^2\right).
\end{align*}
Let $\lambda_i^{*}(\omega)$ denote the $i$-th successive minima of $\Gamma^{*}(\omega)$ with respect to $D^{*}(\delta)$, so that  by Lemma~\ref{lem:transfer}
\begin{align*}
|\Gamma(w)\cap D(\varepsilon)|^2\ll (\delta_1 \dots \delta_d)^2\prod_{i=1}^{d+1-j}\frac{1}{\lambda^{*}_i(\omega)^2}\ll (\delta_1 \dots \delta_d)^2\frac{1}{\lambda^{*}_1(\omega)^{2(d+1-j)}},
\end{align*}
and hence 
\begin{align*}
S_j^{*}\ll (\delta_1\dots \delta_d)^{2}\sum_{\omega\in \Omega^{*}_j}\frac{1}{\lambda^{*}_1(\omega)^{2(d+1-j)}}.
\end{align*}
Let $D^{*}$ denote the convex body 
\begin{align*}
D^{*}=\left\{ (t,s)\in \R^{d}\times \R^{d} \ : \  |t_i|\le \frac{1}{H}, \ |s_i|\le \frac{H}{q}  \right\},
\end{align*}
and let $\mu_1(\omega)$ denote the first successive minima of $\Gamma^{*}(\omega)$ with respect to $D^{*}$. Since 
$$D^{*}(\delta)\subseteq D^{*},$$
we have $$\mu_1(\omega)\le \lambda_1^{*}(\omega),$$
and hence  
\begin{align*}
S_j^{*}\ll (\delta_1\dots \delta_d)^{2}\sum_{\omega\in \Omega^{*}_j}\frac{1}{\mu_1(\omega)^{2(d+1-j)}}.
\end{align*}
We partition
\begin{align*}
\Omega^{*}_j(k)=\{ \omega \in \Omega^{*}_j \ : \ 2^{-(k+1)}\le \mu_1(\omega)<2^{-k} \},
\end{align*}
so that 
\begin{align}
\label{eq:omegastarempty}
\Omega^{*}_j(k)=\emptyset \quad \text{if $2^{k}>H$},
\end{align}
and by the above 
\begin{align*}
S_j^{*}\ll (\delta_1\dots \delta_d)^{2}\sum_{\substack{k \\ 2^k\le H}}2^{2(d+1-j)k}|\Omega^{*}_j(k)|.
\end{align*}
Arguing as in the case of $S_j$, since each nonzero point 
$$\left(\frac{h}{q},\frac{y}{q}\right), \ h\in \cL, \ y\in \Z^{d},$$
belongs to at most one lattice $\Gamma^{*}(\omega)$, we have
\begin{align*}
|\Omega_j(k)|\ll \frac{H^d}{2^{kd}}\left|B\left(\alpha,\frac{1}{2^{k}H}\right)\right|,
\end{align*}
where
$$B(\alpha,\varepsilon)=\left\{ 1\le x \le q : \left\|\frac{\alpha_i x}{q}\right\|\le \frac{1}{2^kH} \quad i=1,\dots,d \right\}.$$  
By~\eqref{eq:A'proper}, ~\eqref{eq:omegastarempty} and Lemma~\ref{lem:shao}
\begin{align*}
|B(\alpha,\varepsilon)|\ll q \prod_{i=1}^{r}\left(\frac{1}{2^kH}+\frac{1}{H^2}\right)\ll \frac{q}{2^{kd}H^d},
\end{align*}
which implies 
\begin{align*}
|\Omega_j(k)|\ll \frac{q}{2^{2kd}},
\end{align*}
and hence 
\begin{align*}
S_j^{*}\ll q(\delta_1\dots \delta_d)^{2}\sum_{\substack{k \\ 2^k\le H}}\frac{1}{2^{k(j-1)}}\ll q(\delta_1\dots \delta_d)^{2}(\log{H}).
\end{align*}
Combining the above with~\eqref{eq:SSj} we get 
\begin{align*}
S^{*}\ll q(\delta_1\dots \delta_d)^{2}(\log{H}),
\end{align*}
and the result follows from~\eqref{eq:thm2s2},~\eqref{eq:OmegaSS*} and~\eqref{eq:Sb123}.
\section{Proof of Corollary~\ref{cor:main2}}
Let $I(\lambda)$ count the number of solutions to 
\begin{align*}
a_1\equiv \lambda a_2 \mod{q},
\end{align*}
with $a_1,a_2\in \cA,$
so that 
\begin{align*}
E(\cA)=\sum_{\lambda}I(\lambda)^2.
\end{align*}
Let $\cA_0$ denote the progression
\begin{align*}
\cA_0=\{ \alpha_1h_1+\dots+\alpha_dh_d \ : \ |h_i|\le H \},
\end{align*}
and suppose $I_0(\lambda)$ counts the number of solutions to the equation
\begin{align*}
a_1\equiv \lambda a_2, \quad a_1,a_2\in \cA_0.
\end{align*}
If $I(\lambda)\neq 0$ then 
\begin{align*}
I(\lambda)\le I_0(\lambda),
\end{align*}
and hence 
\begin{align*}
E(\cA)\le \sum_{\lambda}I_0(\lambda)^2+|\cA|^2\le E(\cA_0)+|\cA|^2,
\end{align*}
and the result follows since $\cA_0$ is the union of at most $2^d$ proper progressions of the form covered by Theorem~\ref{thm:main2}.

\end{document}